\documentclass[a4paper,leqno,11pt]{amsart}
\usepackage{amsfonts,amssymb,verbatim,amsmath,amsthm,latexsym,textcomp,amscd}
\usepackage{latexsym,amsfonts,amssymb,epsfig,verbatim}
\usepackage{amsmath,amsthm,amssymb,latexsym,graphics,textcomp}
\usepackage{graphicx}
\usepackage{color}
\usepackage{stackrel}
\usepackage{tikz}
\usepackage{url}
\usepackage{algpseudocode}
\usepackage{multirow}
\usepackage{enumerate}
\usepackage[mathscr]{euscript}
\usepackage{txfonts}
\usepackage{euscript}
\usepackage[all]{xy}
\usepackage{epic,eepic}
\input xy
\xyoption{all}

\usepackage{hyperref}

\theoremstyle{plain}
\newtheorem{theorem}{Theorem}[section]
\newtheorem{thm}[theorem]{Theorem}
\newtheorem{lemma}[theorem]{Lemma}
\newtheorem{cor}[theorem]{Corollary}
\newtheorem{prop}[theorem]{Proposition}

\newtheorem{example}[theorem]{Example}
\newtheorem{remark}[theorem]{Remark}

\theoremstyle{definition}
\newtheorem{defn}[theorem]{Definition}

\newtheorem{exam}[theorem]{Example}

\begin{document}
\title{On certain classes of algebras in which centralizers are ideals}
\author{Ripan Saha}
\email{ripanjumaths@gmail.com}
\address{Department of Mathematics,
Raiganj University, Raiganj, 733134,
West Bengal, India.}

\author{David A. Towers}
\email{d.towers@lancaster.ac.uk}
\address{Department of Mathematics and Statistics,
Lancaster University, Lancaster LA1 4YF,
England.}

\begin{abstract}
This paper is primarily concerned with studying finite-dimensional anti-commutative nonassociative algebras in which every centralizer is an ideal. These are shown to be anti-associative and are classified over a general field $F$; in particular, they are nilpotent of class at most $3$ and metabelian. These results are then applied to show that a Leibniz algebra over a field of charactersitic zero in which all centralizers are ideals is solvable. 
\end{abstract}
\maketitle

\tableofcontents
\section{Introduction}
Centralizers in algebras have been studied in many papers, including \cite{bi}, \cite{DS19}, \cite{Gor17}, \cite{IP2010}, \cite{smr} (though this paper contains errors as we will point out below) and \cite{zap}. In particular, in \cite{DS19} the authors studied Leibniz algebras in which all of the centralizers are ideals. In this paper we will continue that study for other classes of algebras and answer one of the questions raised in that paper.
\par

Throughout, $A$ will denote a finite-dimensional nonassociative algebra over a field $F$. We will denote the product in $A$ by juxtaposition unless $A$ is a Lie or Leibniz algebra, in which case we will use the usual bracket notation, $[,]$. We will call $A$ \emph{anti-commutative} if $x^2=0$ for all $x\in A$; of course, in such an algebra $xy=-yx$ for all $x,y\in A$. Clearly, all Lie algebras are anti-commutative. In such an algebra, when specifying the non-zero products, we will only specify $xy$, leaving it assumed that $yx=-xy$
\par

The \emph{centralizer} of an element $x\in A$ is the set $$C_A(x)=\{y\in A \mid xy=yx=0 \hbox{ for all } y\in A\}.$$ Following \cite{DS19} we call $A$ a CL-algebra if every centralizer in $A$ is an ideal of $A$. We will say that elements $x, y\in A$ have {\em commutative bonding (CB)} if $xy=0$ implies that $(xz)y=0$ for all $z \in A$. The algebra $A$ is then defined to be a {\em CB-algebra} if every pair of elements of $A$ have commutative bonding. For anti-commutative algebras we will see that these two conditions are equivalent. It is easy to see that not all nonassociative algebras have the CB-property. The smallest such example is the two-dimensional solvable Lie algebra $L$ with basis $e_1,e_2$ and non-zero product $[e_1,e_2]=e_2$. Then $[e_1,e_1]=0$ but $[[e_1,e_2],e_1]=[e_2,e_1]=-e_2$. 
\par

The \emph{Frattini subalgebra}, $F(A)$, of $A$ is the intersection of the maximal subalgebras of $A$. The \emph{Frattini ideal}, $\phi(A)$, of $A$ is the biggest ideal contained in $F(A)$. If $\phi(A)=0$ we say that $A$ is \emph{$\phi$-free}. As pointed out above, there are errors in \cite{smr}. In particular, Proposition 3.4, which claims that a non-abelian Lie algebra $L$ with $\phi(L)\neq 0$ has only one maximal abelian subalgebra, is false. For example, let $L$ be the three-dimensional Heisenberg Lie algebra with basis $e_1,e_2,e_3$ and non-zero product $[e_1,e_2]=e_3$. Then $\phi(L)=L^2=Fe_3\neq 0$ and $Fe_1+Fe_3$, $Fe_2+Fe_3$ are maximal abelian subalgebras of $L$.
\par

 In the following four sections we will consider anti-commutative algebras. In section \ref{sec 2} we introduce some terminology and notation that we use throughout. In section \ref{sec 3} we introduce CB-elements and CB-algebras and show that anti-commutative algebras are CB-algebras if and only if they are anti-associative. It is also shown that anti-commutative algebras are CB-algebras if and only if they are CL-algebras and that the set of such CB-algebras is closed under subalgebras, factor algebras and direct sums.
\par

In section \ref{sec 4} we give a characterisation of all anti-commutative CB-algebras over a general field $F$. In particular, we show that they are all metabelian and nilpotent of class at most three. In section \ref{sec 5} we determine which of the nilpotent Lie algebras of dimension at most six are CB-algebras.
\par

In section \ref{sec 6} we consider the consequences of our earlier results for Leibniz CL-algebras. In particular, we show that, over a field of characteristic zero, all such algebras are solvable, thereby answering a question raised in \cite{DS19}. In the final section we consider group actions on algebras and show that CB-elements are preserved by such actions.
\par

We will denote the subspace spanned by $e_1, \ldots , e_n$ by $Fe_1+\ldots+Fe_n$. Algebra direct sums will be denoted by $\oplus$, whereas $\dot{+}$ will indicate a vector space direct sum.

\section{Preliminaries for anti-commutative algebras} \label{sec 2}
\begin{defn}
An ideal of an algebra $A$ is a subspace $I$ with the property that $IA\subseteq A$.
\end{defn}

Note that, as $A$ is anti-commutative, all ideals are two-sided. 

\begin{exam}
The center $Z(A) = \lbrace x \in A \mid xy=0,~\text{for all}~y\in A \rbrace $ of $A$ is an ideal of $A$.
\end{exam}
\begin{exam}
We define the subalgebras $A^k$ inductively by $A^2 = \text{span}\lbrace xy  \mid x, y\in A \rbrace$, $A^k=A^{k-1}A$ for all $k\geq 3$. Then a straightforward induction argument shows that $A^k\subseteq A^{k-1}$ for all $k\geq 2$, and $A^k$ is an ideal of $A$ for all $k\geq 1$.
\end{exam}
\begin{defn}\label{nilp} $A$ is said to be nilpotent of class $n$ if $A^{n+1}=0$ but $A^n\neq 0$.
\end{defn}
\begin{exam}
Similarly, we define the subalgebras $A^{(k)}$ inductively by $A^{(0)} = A$, $A^{(k)}=A^{(k-1)}A^{(k-1)}$ for all $k\geq 2$. Then a straightforward induction argument shows that $A^{(k)}\subseteq A^{(k-1)}$ for all $k\geq 1$, and $A^{(k)}$ is an subalgebra of $A$ for all $k\geq 1$, but may not be an ideal of $A$.
\end{exam}

\begin{defn}\label{solv} $A$ is said to be solvable if $A^{(n)}=0$ for some $n\geq 1$. If $A^{(2)}=0$ we say that $A$ is metabelian.
\end{defn}
\begin{defn}We define the centralizer of $x$ in $A$ as 
$$C_A(x) = \lbrace y\in A : xy = 0 \rbrace.$$ 
\end{defn}

\begin{defn} We say that $A$ is a CL-algebra if $C_A(x)$ is an ideal of $A$ for all $x\in A$.
\end{defn}

\section{Anti-commutative CB-algebras} \label{sec 3}
\begin{defn}\label{CB-bonding}
Two elements $x, y \in A$ are said to have commutative bonding if $xy=0$ imples $(xz)y=0$ for all $z \in A.$
\end{defn}

\begin{defn}
The anti-commutative algebra $A$ is called a CB-algebra if it satisfies the following property: whenever $x, y \in A$ are such that $xy=0$ then $(xz)y=0$ for all $z \in A.$
\end{defn}
\begin{example}
An algebra $A$ in which $A^2=0$ is automatically a CB-algebra.
\end{example}
\begin{example}
Any nilpotent algebra $A$ of class $2$ is a CB-algebra.
\end{example}

\begin{defn} We define the linear transformation $R_x : A \rightarrow A : a\mapsto ax$. An element $x\in L$ such that $R_x^2 = 0$ is called an absolute zero divisor.
\end{defn}

\begin{remark} In  the  early   1960's   Kostrikin   showed   that   absolute zero divisors  play  a  special  and  very  important role in the theory  of  Lie  algebras over fields of prime characteristic. Since   Lie   algebras   containing  absolute   zero divisors   have   a  degenerate  Killing   form, Kostrikin  called   them  algebras   with  strong  degeneration.
\end{remark}

\begin{thm}\label{equiv} The following are equivalent:
\begin{itemize}
\item[(i)] $A$ is a CB-algebra;
\item[(ii)] every element of $A$ is an absolute zero divisor;
\item[(iii)] $A$ is anti-associative.
\end{itemize}
\end{thm}
\begin{proof} 
$(i) \Rightarrow (ii)$: Let $x,y\in A$. Then $x^2=0$, so $0=-(xy)x=(yx)x$.  As this is true for all $x,y\in A$ we have $R_x^2=0$ for all $x\in A$, so every element of $A$ is an absolute zero divisor, giving (ii). 
\par

\noindent $(ii) \Rightarrow (iii)$: Let $x,y,z$ be arbitrary elements of $A$. Then $(x y)x =0$. Hence
\begin{align*} 0=((x+z)y)(x+z) & = (xy)x+(xy)z+(zy)x+(zy)z  \\
 & = (xy)z+(zy)x. \nonumber 
\end{align*} It follows that $(xy)z=-x(yz)$ and $A$ is anti-associative.
\par

\noindent $(iii) \Rightarrow (i)$: Suppose that $xy=0$. Then $$(xz)y=-y(xz)=(yx)z=-(xy)z=0.$$ Thus, $A$ is a CB-Lie algebra.
\end{proof}

\begin{cor} 
The set of finite-dimensional CB-algebras form a pseudo-variety; that is, they are closed under subalgebras, factor algebras and direct sums.
\end{cor}
\begin{proof} It is clear that anti-associativity is preserved under the taking of subalgebras, factor algebras and direct sums.
\end{proof}

\begin{thm}
The anti-commutative algebra $A$ is a CL-algebra if and only if it is a CB-algebra.
\end{thm}
\begin{proof}
Suppose $C_A(x)$ is an ideal of $A$ for all $x\in A$.
Let for some $x, y \in A$, $xy=0$. This implies that $x \in C_A(y)$. Now for any $z\in A$,  $xz\in C_A(y)$ as $C_A(y)$ is an ideal of $A$.
Thus, $(xz)y=0$. Therefore, $A$ is a CB-algebra.

Conversely, supoose $A$ is a CB-algebra. We need to show $C_A(x)$ is an ideal of $A$ for all $x\in A$. Let $y \in C_A(x)$ and $z\in A$. We need to verify $yz \in C_A(x)$. This clearly follows from the definition of CB-algebra.
\end{proof}

\begin{defn}
An element $z \in A$ is said to have the CB-property if $(xz)y=0$ for all $x \in A$ and $y\in C_A(x)$. We will call such an element a CB-element.
\end{defn}

\begin{remark}
Note that CB-elements are those elements of the algebra which do not break the commutativity between any two elements. For example, $0$ is a CB-element. All the elements in a CB-algebra are CB-elements.
\end{remark}
\begin{lemma} \label{CB-element and anti-ass}
If $z \in A$ be a CB-element then $x(zy)= -(xz)y$ for all $x, y\in A$.
\end{lemma}
\begin{proof}
If $z\in A$ is a CB-element then $(xz)x=0$ for all $x \in A$. Observe that
\begin{align*}
0=((x+y)z)(x+y)& = (xz)x + (yz)x + (xz)y + (yz)y = (yz)x + (xz)y.
\end{align*}
Thus, we have $x(zy)= -(xz)y$.
\end{proof}
\begin{lemma} \label{CB-element and cent}
Let $x \in A,~y \in C_A(x)$. Then $z \in A$ is a CB-element if and only if $zy \in C_A(x)$.
\end{lemma}
\begin{proof}
Let $z \in A$ be a CB-element. For all $x \in A$ and $y \in C_A(x)$, we have
\begin{align*}
(zy)x =-x(zy)=(xz)y = 0,
\end{align*} using anti-commutativity and Lemma \ref{CB-element and anti-ass}.
Thus, $zy\in C_A(x)$.

Conversely, suppose $zy \in C_A(x)$ for some $z\in A$. Interchanging the role of $x$ and $y$, we have $zx \in C_A(y)$. Observe that 
\begin{align*}
(xz)y =0.
\end{align*}
Therefore, $z$ is a CB-element.
\end{proof}
\begin{prop}
The collection $K$ of all CB-elements is a  subalgebra of $A$. Thus, $K$ is a CB- algebra.
\end{prop}
\begin{proof}
As $0\in K$, the set $K$ is non-empty; it is clearly a subspace of $A$. Let $z_1, z_2 \in K$ and $y \in C_A(x)$. We need to show that $z=z_1 z_2 \in K$.

As $z_1, z_2$ are CB-elements, using Lemma (\ref{CB-element and anti-ass}) and  (\ref{CB-element and cent}) for $z_1$ and $z_2$, we have
$$z y = (z_1 z_2)y = -z_1(z_2 y) \in C_A(x).$$
Again using the Lemma \ref{CB-element and cent} for $z$, we get $z= z_1 z_2$ is a CB-element. Thus, the collection $K$ of CB-elements is a  subalgebra, and $K$ is a CB-algebra.
\end{proof}

\begin{prop}
Let $A_1$ and $A_2$ be non-associative algebras with $x^2=0$ for all $x \in A_1, A_2$ and $\phi : A_1 \to A_2$ be a homomorphism. If $z$ is a CB-element of $A_1$ then $\phi(z)$ is a CB-element of $\phi(A_1)$.
\end{prop}
\begin{proof}
Let $z_1 \in A_1$ be a CB-element and $z_2= \phi(z_1)$. Suppose $x_2 \in \phi(A_1)$ and $y_2 \in C_{\phi(A_1)}(x_2)$, so $x_2 = \phi (x_1)$ and $y_2 = \phi(y_1)$ for some $x_1, y_1 \in A_1$. Then
\begin{align*}
(x_2 z_2) y_2 & = (\phi(x_1) \phi (z_1)) \phi(y_1) \\
                           & = \phi (x_1 z_1) \phi(y_1) \\
                           & = \phi ((x_1 z_1) y_1) \\
                           & = \phi(0)=0.
\end{align*}
Therefore, $z_2 = \phi(z_1)$ is a CB-element in $\phi(A_1)$.
\end{proof}

\begin{cor}
Let $A_1$ and $A_2$ be non-associative algebras with $x^2=0$ for all $x \in A_1, A_2$ and $\phi : A_1 \to A_2$ be an isomorphism. If $A_1$ is a CB-algebra then so is $A_2$.
\end{cor}

\section{Classification of anti-commutative CB-algebras} \label{sec 4}
First we show that all anti-commutative CB-algebras are metabelian and nilpotent of index at most three.
\begin{thm}\label{nilp} Let $A$ be a CB-algebra over any field $F$. If $F$ has characteristic different from two then $A^4=0$. Moreover, $A$ is metabelian (that is, $A^{(2)}=0$). If $A$ is a Lie algebra and $F$ has characteristic different from three then $A^3=0$. If $A$ is an associative algebra and $F$ has characteristic different from two then $A^3=0$.
\end{thm}
\begin{proof} Let $x,y,z,w$ be arbitrary elements of $A$. Then $A$ is anti-associative, by Proposition \ref{equiv} and so
 $(xy)z=-x(yz)$.
\par

Now
\[ ((xy)z)w = -(x(yz))w=x((yz)w)=-x(y(zw)).
\] But also
\[   ((xy)z)w =-(xy)(zw)=x(y(zw)).
\]
Hence, if $F$ has characteristic different from two, we have that $x(y(zw)) =0$.
\par

Now suppose that $A$ satisfies the Jacobi identity. Then
\begin{align*}
0=x(yz)+y(zx)+z(xy) & =x(yz)-(zx)y-(xy)z \\
 & =x(yz)+(xz)y+x(yz)=2x(yz)-x(zy)=3x(yz).
\end{align*}
If $F$ has characteristic different from three this implies that $x(yz)=0$, whence $A^3=0$.
\par

If $A$ is associative, then $x(yz)=-(xy)z=(xy)z$ which implies that $x(yz)=0$ if $F$ has characteristic different from two.
\end{proof}

\begin{prop}\label{codim2} Let $A$ be a CB-algebra with $\dim (A/A^2)=2$. Then $A^3=0$.
\end{prop}
\begin{proof} Suppose that $A^3\neq 0$. Then there exist $x,y\in A$ such that $xy\in A^2\setminus Z(A)$. If either $x$ or $y$ is in $A^2$ we have $xy\in A^3\subseteq Z(A)$, since $A^4=0$, by Proposition \ref{nilp}. Hence $A=Fx+Fy+A^2$. Now $x(xy)=y(xy)=0$ implies that $A(xy)=0$, since $A$ is metabelian, by Proposition \ref{nilp} again. It follows that either $x(xy)\neq 0$ or $y(xy)\neq 0$, both of which imply that $A$ is not a CB-algebra, a contradiction. The result follows.
\end{proof}

\begin{defn} A nilpotent Lie algebra $L$ of dimension $n$ is called filiform if $\dim L^i = n-i$ for each $i\geq 2$. 
\end{defn}

Filiform Lie algebras are nilpotent Lie algebras with maximal nilindex: a filiform Lie algebra $L$ of dimension $n$ has $L^n=0$. They were introduced by Vergne in \cite{Ver} and have attracted much attention since then as they have important properties; in particular, every filiform Lie algebra can be obtained by a deformation of the model filiform algebra. However, very few are Lie CB-algebras as the following corollary shows.

\begin{cor} If $L$ is a filiform Lie CB-algebra then $L$ is two- or three-dimensional abelian or the three-dimensional Heisenberg algebra.
\end{cor}
\begin{proof} If $L$ is filiform then $\dim (L/L^2)=2$. It follows from Proposition \ref{codim2} that $\dim L\leq 3$. Hence $L$ is two- or three-dimensional abelian or the three-dimensional Heisenberg algebra.
\end{proof}
\par

In order to classify CB-algebras we follow the ideas in \cite[Theorem 2.4(a)]{ks}. Let $A$ be a CB-algebra, let $B$ be a subspace of $A^2$ which is complementary to $Z(A)$ and let $C$ be a subspace of $A$ which is complementary to $A^2$, so that $A=(Z(A)\dot{+} B)\dot{+} C$. Choose a basis $\{ e_i,\ldots , e_r\}$ for $C$ and put
\begin{itemize}
\item[(1)] $e_i e_j = e_{ij}+z_{ij}$ and $e_i e_{jk} = z_{ijk}$, where $e_{ij}\in B$, $z_{ij}, z_{ijk}\in Z(A)$ for $1\leq i,j,k\leq r$.
\item[(2)] $B^2=0$ (since $A$ is metabelian) and $AZ(A)=0$.
\item[(3)] $e_i^2=0$, $e_{ij} = - e_{ji}$, $z_{ij} = -z_{ji}$ (by anticommutativity) and, for all permutations $\sigma \in S_3$, 
\[ z_{\sigma(i)\sigma(j)\sigma(k)} = sign(\sigma)z_{ijk}
\] (by anti-associativity).
\item[(4)] The set $\{e_{ij} \mid 1\leq i,j\leq r\}$ span $B$, since $e_1,\ldots, e_r$ are the generators of $A$ and $AB\subseteq Z(A)$.
\item[(5)] $\sum \lambda_{jk} e_{ij}=0 \Rightarrow \sum \lambda_{jk} z_{ijk}=0$ for all $1\leq i\leq r$ ($\lambda_{jk}\in F$).
\item[(6)] If $x\in B+C$ then there is $y\in C$ such that $xy\neq 0$.
\end{itemize}

Conversely, if we have three subspaces, $Z,A,B$ such that  $A=(Z\dot{+} B)\dot{+} C$ with $\{ e_i,\ldots , e_r\}$ a basis for $C$ and satisfying (1)-(6) then $A$ is a well-defined algebra, $x^2=0$ for all $x\in A$ and $A$ is anti-associative, by (1), (2), (3). It follows from Theorem \ref{equiv} that $A$ is a CB-algebra.
\par

We have proved the following result.

\begin{thm} An algebra $A$ is a CB-algebra if and only if  it has three subspaces $Z,B,C$ such that $A=(Z\dot{+} B)\dot{+} C$ and satisfying (1)-(6) above.
\end{thm}

\begin{exam}\label{seven} The smallest example of a CB-algebra such that $A^3\neq 0$ will be seven-dimensional in which $C$ is spanned by $e_1,e_2,e_3$, $B$ is spanned by $e_1e_2, e_1e_3, e_2e_3$ and $Z$ is spanned by $e_1(e_2e_3)$. If we denote $e_1e_2$ by $e_4$, $e_1e_3$ by $e_5$, $e_2e_3$ by $e_6$ and $e_1(e_2e_3)$ by $e_7$, this has multiplication
\begin{align*} e_1 e_2 &= e_4 & e_1e_3 &= e_5 &e_2e_3 &= e_6 \\
e_1e_6 &= e_7 & e_2e_5 &= -e_7 & e_3e_4 &= e_7
\end{align*}
Notice that this is Lie algebra if and only if $F$ has characteristic three, and is an associative algebra if and only if $F$ has characteristic two.
\end{exam}

\section{Low dimensional Lie CB-algebras} \label{sec 5}
Here we look at the nilpotent Lie algebras of dimension less than or equal to six to see which of them are CB-algebras. We use the classification over a field of characteristic different from two given in \cite{GR06}, and will employ the same notation as there. For the reader's convenience we list the algebras here. Throughout, $I$ will denote a one-dimensional ideal of $L$.
\begin{prop} Nilpotent Lie algebras of dimensions one or two are abelian; in dimension three there are two non-isomorphic algebras, $L_{3,1}$, which is abelian, and the Heisenberg algebra $L_{3,2}$ with $[e_1,e_2]=e_3$. All of these are CB-algebras.
\end{prop}
\begin{proof} These all have $L^3=0$.
\end{proof}

\begin{prop} In dimension $4$ there are three non-isomorphic nilpotent Lie algebras:
\begin{itemize}
\item $L_{4,1}=L_{3,1}\oplus I$;
\item $L_{4,2}=L_{3,2}\oplus I$; and
\item $L_{4,3}$ with non-zero products $[e_1,e_2]=e_3$, $[e_1,e_3]=e_4$.
\end{itemize} 
Of these, $L_{4,1}$ and $L_{4,2}$ are CB-Lie algebras, but $L_{4,3}$  is not.
\end{prop}
\begin{proof} The first two have $L^3=0$. The third has $[e_1,[e_1,e_2]]\neq 0$, so this is not a CB-algebra.
\end{proof}

\begin{prop} In dimension $5$ there are nine non-isomorphic nilpotent Lie algebras:
\begin{itemize}
\item $L_{5,k}=L_{4,k}\oplus I$ for $k=1,2,3$;
\item $L_{5,4}$: $[e_1,e_2]=e_5, [e_3,e_4]=e_5$;
\item $L_{5,5}$: $[e_1,e_2]=e_3, [e_1,e_3]=e_5, [e_2,e_4]=e_5$;
\item $L_{5,6}$: $[e_1,e_2]=e_3, [e_1,e_3]=e_4, [e_1,e_4]=e_5, [e_2,e_3]=e_5$;
\item $L_{5,7}$: $[e_1,e_2]=e_3, [e_1,e_3]=e_4, [e_1,e_4]=e_5$;
\item $L_{5,8}$: $[e_1,e_2]=e_4$, $[e_1,e_3]=e_5$; and
\item $L_{5,9}$: $[e_1,e_2]=e_3, [e_1,e_3]=e_4, [e_2,e_3]=e_5$.
\end{itemize}
Of these, $L_{5,1}, L_{5,2}, L_{5,4}$ and $L_{5,8}$ are CB-Lie algebras, but the others are not.
\end{prop}
\begin{proof} $L_{5,1}, L_{5,2}, L_{5,4}$ and $L_{5,8}$ all have $L^3=0$. 
\par

$L_{5,3}$ is not a CB-Lie algebra because $L_{4,3}$ isn't one. 
\par

$L_{5,5}$, $L_{5,6}$, $L_{5,7}$ and $L_{5,9}$ all have $[e_1, [e_1,e_2]]\neq 0$, and so they are not CB-Lie algebras.
\end{proof}

\begin{prop} In dimension $6$ we get the following nilpotent Lie algebras:
\begin{itemize}
\item $L_{6,k}=L_{5,k}\oplus I$ for $k=1,\ldots,9$;
\item $L_{6,10}$:  $[e_1,e_2]=e_3, [e_1,e_3]=e_6, [e_4,e_5]=e_6$;
\item $L_{6,11}$:  $[e_1,e_2]=e_3, [e_1,e_3]=e_4, [e_1,e_4]=e_6, [e_2,e_3]=e_6, [e_2,e_5]=e_6$;
\item $L_{6,12}$:  $[e_1,e_2]=e_3, [e_1,e_3]=e_4, [e_1,e_4]=e_6, [e_2,e_5]=e_6$;
\item $L_{6,13}$:  $[e_1,e_2]=e_3, [e_1,e_3]=e_5, [e_2,e_4]=e_5, [e_1,e_5]=e_6, [e_3,e_4]=e_6$;
\item $L_{6,14}$:  $[e_1,e_2]=e_3, [e_1,e_3]=e_4, [e_1,e_4]=e_5, [e_2,e_3]=e_5, [e_2,e_5]=e_6$, $[e_3,e_4]=-e_6$;
\item $L_{6,15}$:  $[e_1,e_2]=e_3, [e_1,e_3]=e_4, [e_1,e_4]=e_5, [e_2,e_3]=e_5, [e_1,e_5]=e_6$, $[e_2,e_4]=e_6$;
\item $L_{6,16}$:  $[e_1,e_2]=e_3, [e_1,e_3]=e_4, [e_1,e_4]=e_5, [e_2,e_5]=e_6, [e_3,e_4]=-e_6$;
\item $L_{6,17}$:  $[e_1,e_2]=e_3, [e_1,e_3]=e_4, [e_1,e_4]=e_5, [e_1,e_5]=e_6, [e_2,e_3]=e_6$:
\item $L_{6,18}$:  $[e_1,e_2]=e_3, [e_1,e_3]=e_4, [e_1,e_4]=e_5, [e_1,e_5]=e_6$:
\item $L_{6,19}(\epsilon)$:  $[e_1,e_2]=e_4, [e_1,e_3]=e_5, [e_2,e_4]=e_6, [e_3,e_5]=\epsilon e_6$;
\item $L_{6,20}$:  $[e_1,e_2]=e_4, [e_1,e_3]=e_5, [e_1,e_5]=e_6, [e_2,e_4]=e_6$;
\item $L_{6,21}(\epsilon)$:  $[e_1,e_2]=e_3, [e_1,e_3]=e_4, [e_2,e_3]=e_5, [e_1,e_4]=e_6, [e_2,e_5]=\epsilon e_6$;
\item $L_{6,22}(\epsilon)$:  $[e_1,e_2]=e_5, [e_1,e_3]=e_6, [e_2,e_4]=\epsilon e_6, [e_3,e_4]= e_5$;
\item $L_{6,23}$: $[e_1,e_2]=e_3, [e_1,e_3]=e_5, [e_1,e_4]=e_6, [e_2,e_4]=e_5$;
\item $L_{6,24}(\epsilon)$:  $[e_1,e_2]=e_3, [e_1,e_3]=e_5, [e_1,e_4]=\epsilon e_6, [e_2,e_3]= e_6, [e_2,e_4]=e_5$;
\item $L_{6,25}$: $[e_1,e_2]=e_3, [e_1,e_3]=e_5, [e_1,e_4]=e_6$; and
\item $L_{6,26}$: $[e_1,e_2]=e_4, [e_1,e_3]=e_5, [e_2,e_3]=e_6$,
\end{itemize}
where $\epsilon \in F$. Of these, $L_{6,1}, L_{6,2}, L_{6,4}, L_{6,8}, L_{6,22}(\epsilon)$ and $L_{6,26}$ are CB-Lie algebras, but the others are not.
\end{prop}
\begin{proof}  $L_{6,1}, L_{6,2}, L_{6,4}, L_{6,8}, L_{6,22}(\epsilon)$ and $L_{6,26}$ all have $L^3=0$.
\par

$L_{6,3}, L_{6,5}, L_{6,6}, L_{6,7}$ and $L_{6,9}$ are not CB-Lie algebras because the corresponding $L_{5,k}$ isn't one.
\par

$L_{6,10}, L_{6,11}, L_{6,12}, L_{6,13}, L_{6,14}, L_{6,15}, L_{6,16}, L_{6,17}, L_{6,18}, L_{6,21}(\epsilon), L_{6,23}, L_{6,24}(\epsilon)$ and $L_{6,25}$ all have $[e_1,[e_1,e_2]]\neq 0$ and so are not CB-Lie algebras.
\par

$L_{6,19}$ and $L_{6,20}$ have $[[e_1,e_2],e_2]\neq 0$ and so are not CB-Lie algebras.
\end{proof}
\medskip

Note that the above results confirm that the seven-dimensional Lie algebra given in the Example \ref{seven} is the smallest nilpotent CB-Lie algebra $L$ with $L^3\neq 0$.
\section{Leibniz algebras} \label{sec 6}
An algebra $L$ over a field $F$ is called a {\em Leibniz algebra} if, for every $x,y,z \in L$, we have
\[  [x,[y,z]]=[[x,y],z]-[[x,z],y].
\]
In other words the right multiplication operator $R_x : L \rightarrow L : y\mapsto [y,x]$ is a derivation of $L$. As a result such algebras are sometimes called {\it right} Leibniz algebras, and there is a corresponding notion of {\it left} Leibniz algebras, which satisfy
\[  [x,[y,z]]=[[x,y],z]+[y,[x,z]].
\]
Clearly the opposite of a right (left) Leibniz algebra is a left (right) Leibniz algebra, so, in most situations, it doesn't matter which definition we use.
\par
 
Every Lie algebra is a Leibniz algebra and every Leibniz algebra satisfying $[x,x]=0$ for every element is a Lie algebra. The usual definitions for subalgebra, right (left) ideal, ideal, homomorphism apply for Leibniz algebras. Put $I=\langle\{x^2:x\in L\}\rangle$. Then 
\begin{align} [y,x^2]=&[[y,x],x]-[[y,x],x]=0 \hbox{ and } \nonumber \\
[x^2,y]=&[x,[x,y]]+[[x,y],x]=(x+[x,y])^2-x^2-[x,y]^2\in I, \nonumber
\end{align}
so $I$ is an ideal; in fact, $I$ is the smallest ideal of $L$ such that $L/I$ is a Lie algebra; $L/I$ is sometimes called the {\em liesation} of $L$.

In \cite{DS19} the authors asked whether Leibniz CL-algebras are solvable. For Leibniz algebras over a field $F$ of characteristic zero, using our previous results, we can answer this in the affirmative.

\begin{thm} Let $L$ be a Leibniz CL-algebra over a field of characteristic zero. Then $L$ is solvable.
\end{thm}
\begin{proof} By Levi's Theorem (see \cite{barnes}) we have $L=R\dot{+} S$ where $R$ is the radical and $S$ is a semisimple Lie subalgebra of $L$. Now $S$ is a Lie CB-algebra and so is nilpotent, by Theorem \ref{nilp}. It follows that $S=0$, whence the result.
\end{proof}

Over more general fields we have the following results.

\begin{thm} If $L$ is a $\phi$-free Leibniz CL-algebra over any field then $L^{(3)}=0$.
\end{thm}
\begin{proof} Since $L$ is $\phi$-free we have $L=I\dot{+} B$ where $B$ is a subalgebra of $L$, by \cite[Lemma 7.2]{frat}. Now $B$ is a Lie CB-algebra and so $B^{(2)}=0$, by Theorem \ref{nilp}. It follows that $L^{(3)}\subseteq I^2 =0$, as claimed.
\end{proof}

The following is proved in \cite[Lemma 1]{jp}.

\begin{lemma}\label{symm}
If a right Leibniz algebra $L$ is also a left Leibniz algebra, then, for all $x,y,z\in L$,
\begin{itemize}
\item[(i)] $[[x,y],z]+[z,[x,y]]=0$; and
\item[(ii)] $2[[x,y],[x,y]]=0$.
\end{itemize}
\end{lemma}

\begin{defn} We call $L$ a symmetric Leibniz algebra if it is both a right and left Leibniz algebra and $[[x,y],[x,y]]=0$. Note that if $F$ has characteristic different from two the added identity is not needed because of Lemma \ref{symm}(ii).
\end{defn}

\begin{prop}\label{symmder} If $L$ is a symmetric Leibniz algebra over any field then $L^2$ is a Lie algebra.
\end{prop}
\begin{proof} This follows easily from Lemma \ref{symm} and remarks made at the beginning of this section.
\end{proof}

\begin{cor}  If $L$ is a symmetric Leibniz CL-algebra over any field then $L^{(3)}=0$. 
\end{cor}
\begin{proof} It follows from Proposition \ref{symmder} that $L^2$ is a Lie CB-algebra, so $L^{(3)}=(L^{(1)})^{(2)}=0$, by Theorem \ref{nilp}.
\end{proof}

\section{Group actions on CB-algebras} \label{sec 7}
In this section, we study a finite group action on algebras and show that CB-elements is preserved under the group action.
\begin{defn}\label{definition-group-action}
Let $A$ be an algebra and $G$ be a finite group. We say the group $G$ is  acting on $A$  if there exists a function 
$$\phi : G\times A \rightarrow A,~~ (g, x) \mapsto \phi (g, x) = gx$$ satisfying the following conditions.
\begin{enumerate}
\item For each $g\in G$ the map $x\mapsto gx,$ denoted by $\psi_g$ is linear.
\item $ex= x$ for all $x \in A$, where $e \in G$ is the group identity.
\item $g_1(g_2x) = (g_1g_2)x$ for all $g_1, g_2 \in G$ and $x \in A$.
\item $g(x y) = (gx) (gy)$ for all $g\in G$ and $x, y \in A.$
\end{enumerate}  
\end{defn}

The above definition of group action can be equivalently stated as follows:
\begin{prop}
A finite group $G$ acts on  $A$ if and only if there exists a group homomorphism 
$$\psi : G \rightarrow \text{Aut} (A),~~g \mapsto \psi(g)=\psi_g,$$ from the group $G$ to the automorphism group of $A$, where $\psi_g (x) = gx$ is the left translation by $g.$  
\end{prop}
\begin{remark}
Let $G$ be a finite group and ${F}[G]$ be the associated group algebra. If $G$ acts on a algebra $A$  then $A$ may be viewed as a ${F}[G]$-module.
\end{remark}

\begin{thm}
Let $x\in A$ be a CB-element. Then $x$ is preserved under the action of $G$, that is, $gx$ is also a CB-element for all $g \in G$.
\end{thm}
\begin{proof}
Let $A$ be equipped with an action of a group $G$. Let $z \in A$ be a CB-element. For all $x \in A,~ y\in C_A(x), ~ g\in G$, we have
\begin{align} \label{CB-element preserve}
(x (gz)) y = g (((g^{-1}x) z) g^{-1}y) =0.
\end{align}
Observe that in the above computation, we have used the fact that $g^{-1}y \in C_A(g^{-1}x)$ as $(g^{-1}x) (g^{-1}y) = g^{-1}(x y) =g^{-1}0=0$. Equation (\ref{CB-element preserve}) shows $gz$ are also CB-elements for all $g \in G$. Therefore, CB-elements are preserved under the group action.
\end{proof}

Suppose the given algebra $A$ is equipped with an action of a finite group $G$. Let $x \in A$, we define an orbit of $x$ under the action of $G$ as follows:
$$G(x) = \lbrace gx \mid g \in G \rbrace \subseteq A.$$
It is easy to check orbits $G(x)$ and $G(y)$ of any two points $x, y \in A$ are either equal or disjoint. Note that if $x \in A$ is a CB-element then all the elements in the corresponding orbit $G(x)$ are also CB-elements. Let $S$ be the collection of all CB-elements of $A$.
\begin{prop}
The set $B = \bigcup_{x \in S} G(x)$ is a CB-subalgebra of $A$.
\end{prop}
\begin{proof}
Note that $B$ is non-empty as $0\in B$. It is clear from the construction of $B$ that every elements of $B$ are CB-elements. We only need to show that $B$ is a subalgebra of $A$. Let $b_1, b_2 \in B$. Then $b_1 = g_1 x$ and $b_2 = g_2 y$ for some $g_1, g_2 \in G$ and $x, y\in A$. Observe that
$$ab=(g_1 x)(g_2 y) = g_1 (x (g_1^{-1}g_2 y)) \in B.$$
As the group action is linear, this proves $B$ is a subalgebra of $A$. Therefore, $B$ is a CB-algebra.
\end{proof}
\mbox{ }\\

\providecommand{\bysame}{\leavevmode\hbox to3em{\hrulefill}\thinspace}
\providecommand{\MR}{\relax\ifhmode\unskip\space\fi MR }
\providecommand{\MRhref}[2]{%
  \href{http://www.ams.org/mathscinet-getitem?mr=#1}{#2}
}
\providecommand{\href}[2]{#2}

\mbox{ } \\
\end{document}